\documentclass[12pt]{article}

\usepackage{amssymb}
\usepackage{amsmath,amsfonts}
\usepackage{amsthm,bbm,eufrak,upgreek}
\usepackage{mathrsfs,booktabs}
\usepackage{amsxtra,mathtools}
\usepackage[colorlinks=True,linkcolor=blue,anchorcolor=blue,citecolor=blue,filecolor=blue,CJKbookmarks=True]{hyperref}
\usepackage[a4paper,left=2.6cm,right=2.6cm,top=3.2cm,bottom=3.2cm]{geometry}
\usepackage[hang,flushmargin]{footmisc}

\newcommand{\1}{{\bf 1}}
\newcommand{\C}{{\mathbb C}}
\newcommand{\Z}{{\mathbb Z}}

\newcommand{\w}{\omega}

\newcommand{\al}{\alpha}

\DeclareMathOperator{\End}{End}

\DeclareMathOperator{\wt}{wt}

\def\de{\delta}

\def \<{\langle}
\def \>{\rangle}

\newtheorem{thm}{Theorem}[section]
\newtheorem{prop}[thm]{Proposition}
\newtheorem{lem}[thm]{Lemma}
\newtheorem{cor}[thm]{Corollary}
\newtheorem{rmk}[thm]{Remark}
\newtheorem{defn}[thm]{Definition}

\begin{document}

 \begin{center}
	{{\bf A class of vertex operator algebras generated by Virasoro vectors}\footnote{Supported by China NSF grant No.12171312}}
\end{center}

\begin{center}
	{Runkang Feng
		\\
		School of Mathematical  Sciences, Shanghai Jiao Tong University\\
		Shanghai 200240, China}
\end{center}

\let\thefootnote\relax

\begin{abstract}
In this paper, we study a class of simple  OZ-type vertex operator algebras $V$ generated by simple Virasoro vectors $\omega^{ij}=\omega^{ji}$, $1\leq i<j\leq n$, $n\geq 3$. We prove that  $V$ is uniquely determined by its Griess algebra $V_2$. The automorphism group of $V$ is also determined. Furthermore, we give the necessary conditions for $V$ to be unitary. 
\end{abstract}


\section{Introduction}
Let $V$ be a vertex operator algebra. An element $e\in V_2$ is called a Virasoro vector of $V$ if the vertex operator subalgebra $\<e\>$ generated by $e$ is isomorphic to a Virasoro vertex operator algebra
with central charge $c$. If $\<e\>\cong L(c,0)$ is simple, $e$ is called a simple Virasoro vector of $V$. If $\<e\>$ is isomorphic to the simple Virasoro algebra $L(1/2,0)$ with central charge $c=1/2$, $e$ is called an Ising vector of $V$.  If further $V$ as an $\<e\>$-module,  $V$ is a direct sum of $\<e\>$-modules isomorphic to $L(1/2,0)$ or $L(1/2, 1/2)$, $e$ is called an Ising vector of $\sigma$-type \cite{Mi96}. Vertex operator algebras generated by Ising vectors of $\sigma$-type are interesting not only from the view point of VOA theory, but also from the point of view of finite group theory since each Ising vector defines an involution, usually called Miyamoto involution \cite{Mi96}. If all Ising vectors are of $\sigma$-type, the group generated by
Miyamoto involutions forms a 3-transposition group.
When the VOAs have compact real forms, the 3-transposition groups obtained in this manner are of symplectic type \cite{CH} and are completely classified in \cite{Ma}.
In \cite{Ma}, a complete list of 3-transposition groups generated by Miyamoto involutions associated with Ising vectors of $\sigma$-type as well as examples of VOAs realizing those groups is presented. A vertex operator algebra is called of CFT type, if $V=\oplus_{n=0}^{\infty}V_n$ and $V_0=\mathbb{C}{\mathbf{1}}$. If further $V_1=0$, then $V$ is called a  moonshine type or OZ-type vertex operator algebra. OZ type vertex operator algebras generated by Ising vectors of $\sigma$-type have been extensively studied in \cite{LS08}, \cite{JL16}, \cite{JLY19}, \cite{JLY23}. Complete classification and characterization of OZ-type vertex operator algebras were established in \cite{JLY23}. Matsuo's  classification of the center-free  3-transposition groups of symplectic type realizable by an OZ-type VOA generated by Ising vectors of $\sigma$-type without the assumption that the VOA carries a positive-definite Hermitian form was also accomplished in \cite{JLY23}. Stimulated by the above works, we study in this paper a class of simple  OZ-type vertex operator algebras $V$ generated by simple Virasoro vectors $\omega^{ij}=\omega^{ji}$, $1\leq i<j\leq n$, of central charge $c_m$ satisfying conditions similar to those of  Ising vectors of $\sigma$-type, where $c_m=1-\frac{6}{(m+2)(m+3)}$, $m\in\Z_{\geq 2}$.  Concretely, we assume that for distinct $1\leq i,j,k,l\leq n$, 
\begin{equation}\label{e1-1}
\omega^{ij}_1\omega^{jl}=\frac{h^{(m)}_{m+1,1}}{2}(\omega^{ij}+\omega^{jl}-\omega^{il}),  
\end{equation}
and 
\begin{equation}\label{e1-2}
\omega^{ij}_p\omega^{kl}=0, \ p\geq 0, \  \omega^{ij}_3\omega^{ij}=\frac{c_{m}}{2}{\bf 1}, \ \   
\omega^{ij}_3\omega^{jl}=\frac{c_mh_{m+1,1}^{(m)}}{8}{\bf 1},
\end{equation}
where for $1\leq r\leq m+1, 1\leq s\leq m+2$, $h^{(m)}_{r,s}=\frac{[r(m+3)-s(m+2)]^2-1}{4(m+2)(m+3)}$.  Notice that 	if  $m=1$,  $V$ is the  class of vertex operator algebras   generated by  Ising vectors of $\sigma$-type  studied in \cite{JLY19},  \cite{JL16}, \cite{LS08}, \cite{LSY07},  \cite{LY2}, \cite{Ma}, etc. If $m=2$, $n=3$, there is an explicit realization of $V$ in \cite{LS08}. For general  $m\geq 2$, $n\geq 4$, it seems that no explicit realizations of $V$ have been given up to now. It would be interesting to give concrete examples of $V$ or to study whether such $V$ exists or not for $m\geq 2$, $n\geq 4$.  Another  interesting direction  is  the 
 study of structures of  such a class of vertex operator algebras $V$.  In this paper, we focus on the second problem. 
We  first give a linear spanning set of $V$(Theorem \ref{thm1}), and prove that $V$ is uniquely determined by the structure of its Griess algebra $V_2$ (Theorem \ref{thm2}). We next give the automorphism group of $V$  by Theorem \ref{thm4}. Finally,  we give  the necessary condition for $V$ to be unitary. 

The paper is organized as follows. In Section 2, we recall some basic concepts and facts on vertex operator algebras. Section 3 is dedicated to the structure of the vertex operator algebras $V$ we study. In Section 4, we give the  necessary conditions for $V$ to be unitary.

\section{Preliminaries}
\def\theequation{2.\arabic{equation}}
\setcounter{equation}{0}

In this section, we review notions of vertex operator algebras, notions of modules, and some basic facts \cite{FF84},  \cite{Bor86}, \cite{DLM98},
\cite{FLM88}, \cite{FZ92}, \cite{LL04}, \cite{Z96}, \cite{KPP08}, \cite{DL14}.  We also recall intertwining
operators,  fusion rules and some consequences following  \cite{ADL05}, \cite{Li99}, \cite{DMZ94}, \cite{FHL93},  \cite{LS08}, \cite{M98}, \cite{W93}.
\begin{defn}
	A {\em vertex operator algebra} $(V, {\bf 1}, \omega, Y)$ is a $\Z$-graded vector space
	$V=\bigoplus_{n\in \Z} V_n$ such that $\dim V_n<\infty$, and $V_n=0$ for $n<<0$, equipped with a linear map $$Y:V\to(\End
	V)[[z,z^{-1}]],\,a\mapsto Y(a,z)=\sum_{n\in \Z}a_nz^{-n-1}$$
	for
	$a\in V$ such that
	
	(1) For $u,v\in V$,
	$u_nv=0$,  for $n \gg 0.$
	
	(2) There are two
	distinguished vectors, the {\it vacuum vector} $\1\in V_0$ and the
	{\it conformal vector} $\omega \in V_2$ such that 
	$$Y(\1,z)={\rm	id}_{V}, \ \lim_{z\rightarrow 0}Y(u,z){\1}=u, \ {\rm for} \ u\in V$$
	and
	$$
	[L(m),L(n)]=(m-n)L(m+n)+\delta_{m+n,0}\frac{m^3-m}{12}c,
	$$
	where
	$Y(\w,z)=\sum_{n\in\Z}L(n)z^{-n-2}$, $c\in{\C}$ is called the central charge of $V$.
	
	(3) $L(0)_{V_n}=n{\rm id}_{V_n}  $.
	
	(4) The Jacobi identity holds:
	\begin{eqnarray}
	& &z_0^{-1}\de \left({z_1 - z_2 \over
		z_0}\right)Y(u,z_1)Y(v,z_2)-
	z_0^{-1} \de \left({z_2- z_1 \over -z_0}\right)Y(v,z_2)Y(u,z_1) \nonumber \\
	& &\ \ \ \ \ \ \ \ \ \ =z_2^{-1} \de \left({z_1- z_0 \over
		z_2}\right)Y(Y(u,z_0)v,z_2).
	\end{eqnarray}
\end{defn}

\begin{defn} A weak $V$-module is a vector space $M$ equipped
	with a linear map
	$$
	\begin{array}{ll}
	Y_M: & V \rightarrow ({\rm End}M)[[z,z^{-1}]]\\
	& v \mapsto Y_M(v,z)=\sum_{n \in \Z}v_n z^{-n-1},\ \ v_n \in {\rm End}M
	\end{array}
	$$
	satisfying the following:
	
	(1) $v_nw=0$ for $n \gg 0$ where $v \in V$ and $w \in M$,
	
	(2) $Y_M( {\textbf 1},z)={\rm id}_M$,
	
	(3) The Jacobi identity holds:
	\begin{eqnarray}
	& &z_0^{-1}\de \left({z_1 - z_2 \over
		z_0}\right)Y_M(u,z_1)Y_M(v,z_2)-
	z_0^{-1} \de \left({z_2- z_1 \over -z_0}\right)Y_M(v,z_2)Y_M(u,z_1) \nonumber \\
	& &\ \ \ \ \ \ \ \ \ \ =z_2^{-1} \de \left({z_1- z_0 \over
		z_2}\right)Y_M(Y(u,z_0)v,z_2).
	\end{eqnarray}
\end{defn}


\begin{defn}
	An admissible $V$-module is a weak $V$ module  which carries a
	$\Z_+$-grading $M=\bigoplus_{n \in \Z_+} M(n)$, such that if $v \in
	V_r$ then $v_m M(n) \subseteq M(n+r-m-1).$
\end{defn}

\begin{defn}
	An ordinary $V$-module is a weak $V$ module which carries a
	$\C$-grading $M=\bigoplus_{\l \in \C} M_{\l}$, such that:
	
	(1) $\dim(M_{\l})< \infty,$
	
	(2) $M_{\l+n}=0$ for fixed $\l$ and $n \ll 0,$
	
	(3) $L(0)w=\l w=\wt(w) w$ for $w \in M_{\l}$ where $L(0)$ is the
	component operator of $Y_M(\omega,z)=\sum_{n\in\Z}L(n)z^{-n-2}.$
\end{defn}

\begin{rmk} \ It is easy to see that an ordinary $V$-module is an admissible one. If $W$  is an
	ordinary $V$-module, we simply call $W$ a $V$-module.
\end{rmk}

We call a vertex operator algebra rational if the admissible module
category is semisimple. We have the following result from
\cite{DLM98} (also see \cite{Z96}).

\begin{thm}\label{tt2.1}
	If $V$ is a  rational vertex operator algebra, then $V$ has finitely
	many irreducible admissible modules up to isomorphism and every
	irreducible admissible $V$-module is ordinary.
\end{thm}

Recall from \cite{Z96} that a vertex operator algebra is called $C_2$-cofinite if $C_2(V)$ has finite codimension where $C_2(V)=\langle u_{-2}v|u,v\in V\rangle$.

\begin{rmk} If $V$ is a vertex operator algebra satisfying $C_{2}$-cofinite
	property, $V$ has only finitely many irreducible admissible modules
	up to isomorphism \cite{DLM98}, \cite{Li99}, \cite{Z96}.
\end{rmk}
\begin{defn}
	Let $V$ be a vertex operator algebra. $V$ is called regular if every weak $V$-module is a direct sum of ordinary $V$-modules. $V$ is called a moonshine-type or OZ-type  vertex operator algebra if $V=\bigoplus_{n\in\Z_{\geq 0}}V_n$ such that $V_0=\C\1$ and $V_1=0$.
\end{defn}

Let $(V,{\bf 1},\omega,Y)$ be a vertex operator algebra. A vertex operator subalgebra $(U,{\bf 1},\omega',Y)$  of $V$ is a graded subspace $U$ of $V$ which has  a vertex operator algebra structure  such that the operations  agree with the restriction of
those of $V$ and that $U$ shares the same vacuum vector with $V$. If the conformal vector $\omega'$ of $U$ is equal to $\omega$, we call $U$ a conformal vertex operator subalgebra of $V$.

Let $(U,{\bf 1},\omega',Y)$ be a vertex operator subalgebra of $(V,{\bf 1},\omega,Y)$. Set
$$
C_{V}(U)=\{v\in V \ | \ [Y(u,z_1),Y(v,z_2)]=0, \ \forall u\in U\}.
$$
$C_V(U)$ is called the commutant of $U$ in $V$. We have the following result from \cite{FZ92} and \cite{LL04}.
\begin{lem} Let $(V,{\bf 1},\omega,Y)$ be a vertex operator algebra such that $V_n=0$ for $n<0$ and $V_0=\C{\bf 1}$. Let $(U,{\bf 1},\omega',Y)$ be a vertex operator subalgebra of $V$. Assume that
	$$
	\omega'\in U\cap V_2, \ L(1)\omega'=0.
	$$
	Denote $\omega''=\omega-\omega'$. Then $(C_V(U),{\bf 1},\omega'',Y)$ is a vertex operator subalgebra of $V$ of central charge $c_V-c_U$.
\end{lem}
Let $V$ be a vertex operator algebra.  An associative algebra $A(V)$
has been introduced and studied in \cite{Z96}. It turns out that $A(V)$ is
very useful in representation theory for vertex
operator algebras. One can use $A(V)$ not only to classify the
irreducible admissible modules \cite{Z96}, but also to compute the
fusion rules using $A(V)$-bimodules \cite{FZ92}, \cite{Li99}. We will  review
the definition of $A(V)$ and some properties of $A(V)$
from \cite{Z96}, \cite{FZ92}.

As a vector space, $A(V)$ is a quotient space of $V$ by $O(V)$,
where $O(V)$ denotes the linear span of elements
\begin{equation}\label{1.1}
u\circ v={\rm Res}_z(Y(u,z)\frac{(z+1)^{{\rm wt}\, u}}{z^2}v)
=\sum_{i\geq 0}{{\rm wt}\,u\choose i} u_{i-2}v
\end{equation}
for $u,v\in V$ with $u$ being homogeneous. Product in $A(V)$ is
induced from the multiplication
\begin{equation}\label{1.2}
u*v={\rm Res}_z(Y(u,z)\frac{(z+1)^{{\rm wt}\,u}}{z}v)
=\sum_{i\geq 0}{{\rm wt}\,u\choose i}u_{i-1}v
\end{equation}
for $u,v\in V$ with $u$ being homogeneous.  $A(V)=V/O(V)$ is an
associative algebra with identity ${\bf 1}+O(V)$ and with
$\omega+O(V)$ being in the center of $A(V)$. The most important
result about $A(V)$ is that for any admissible $V$-module
$M=\oplus_{n\geq 0}M(n)$ with $M(0)\ne 0,$ $M(0)$ is an
$A(V)$-module such that $v+O(V)$ acts as $o(v)$ where $o(v)=v_{wt
	v-1}$ for homogeneous $v.$

We now recall the  notion of intertwining operators and fusion
rules from \cite{FHL93}.
\begin{defn}
	Let $V$ be a vertex operator algebra and $M^1$, $M^2$, $M^3$ be weak $V$-modules. An intertwining
	operator $\mathcal {Y}( \cdot , z)$ of type $\left(\begin{tabular}{c}
	$M^3$\\
	$M^1$ $M^2$\\
	\end{tabular}\right)$ is a linear map$$\mathcal
	{Y}(\cdot, z): M^1\rightarrow Hom(M^2, M^3)\{z\}$$ $$v^1\mapsto
	\mathcal {Y}(v^1, z) = \sum_{n\in \mathbb{C}}{v_n^1z^{-n-1}}$$
	satisfying the following conditions:
	
	(1) For any $v^1\in M^1, v^2\in M^2$, and $\lambda \in \mathbb{C},
	v_{n+\lambda}^1v^2 = 0$ for $n\in \mathbb{Z}$ sufficiently large.
	
	(2) For any $a \in V, v^1\in M^1$,
	$$z_0^{-1}\delta(\frac{z_1-z_2}{z_0})Y_{M^3}(a, z_1)\mathcal
	{Y}(v^1, z_2)-z_0^{-1}\delta(\frac{z_1-z_2}{-z_0})\mathcal{Y}(v^1,
	z_2)Y_{M^2}(a, z_1)$$
	$$=z_2^{-1}\delta(\frac{z_1-z_0}{z_2})\mathcal{Y}(Y_{M^1}(a, z_0)v^1, z_2).$$
	
	(3) For $v^1\in M^1$, $\dfrac{d}{dz}\mathcal{Y}(v^1,
	z)=\mathcal{Y}(L(-1)v^1, z)$.
\end{defn}
All of the intertwining operators of type $\left(\begin{tabular}{c}
$M^3$\\
$M^1$ $M^2$\\
\end{tabular}\right)$ form a vector space denoted by $I_V\left(\begin{tabular}{c}
$M^3$\\
$M^1$ $M^2$\\
\end{tabular}\right)$. The dimension of $I_V\left(\begin{tabular}{c}
$M^3$\\
$M^1$ $M^2$\\
\end{tabular}\right)$ is called the
fusion rule of type $\left(\begin{tabular}{c}
$M^3$\\
$M^1$ $M^2$\\
\end{tabular}\right)$ for $V$.

We now have the following result which was  proved in
\cite{ADL05}.

\begin{thm}\label{n2.2c}
	Let $V^1, V^2$ be rational vertex operator algebras. Let $M^1 , M^2,
	M^3$ be $V^1$-modules and  $N^1, N^2, N^3$ be $V^2$-modules such
	that
	$$\dim I_{V^1}\left(\begin{tabular}{c}
	$M^3$\\
	$M^1$ $M^2$\\
	\end{tabular}\right)< \infty , \dim I_{V^2}\left(\begin{tabular}{c}
	$N^3$\\
	$N^1$ $N^2$\\
	\end{tabular}\right)< \infty.$$
	Then the linear map
	$$\sigma: I_{V^1}\left(\begin{tabular}{c}
	$M^3$\\
	$M^1$ $M^2$\\
	\end{tabular}\right)\otimes I_{V^2}\left(\begin{tabular}{c}
	$N^3$\\
	$N^1$ $N^2$\\
	\end{tabular}\right)\rightarrow I_{V^1\otimes V^2}\left( \begin{tabular}{c}
	$M^3\otimes N^3$\\
	$M^1\otimes N^1$ $M^2\otimes N^2$\\
	\end{tabular}\right)$$
	$$\mathcal{Y}_1( \cdot, z)\otimes \mathcal{Y}_2(\cdot, z)\mapsto (\mathcal{Y}_1\otimes
	\mathcal{Y}_2)(\cdot, z)$$\\is an isomorphism, where
	$(\mathcal{Y}_1\otimes \mathcal{Y}_2)(\cdot, z)$ is defined by
	$$(\mathcal{Y}_1\otimes \mathcal{Y}_2)(\cdot, z)(u^1\otimes v^1,z)(u^2\otimes v^2)= \mathcal{Y}_1(u^1,z)u^2\otimes \mathcal{Y}_2(v^1,z)v^2.$$
\end{thm}

For $m,r,s\in\Z_{+}$ such that $1\leq r\leq m+1$ and $1\leq s\leq m+2$,  let
$$c_{m}=1-\frac{6}{(m+2)(m+3)}$$ and $$h^{(m)}_{r,s}=\frac{[r(m+3)-s(m+2)]^2-1}{4(m+2)(m+3)}.$$

Recall that $L(c_{m},0)$, $m=1,2,\cdots$ are rational Virasoro vertex operator algebras and $L(c_{m}, h^{(m)}_{r,s})$, $1\leq r\leq m+1, 1\leq s\leq m+2$,  are exactly all the irreducible modules of $L(c_{m},0)$  \cite{W93}, \cite{DMZ94}, \cite{FF84}. The fusion rules for these modules are given in \cite{W93} (see also \cite{DMZ94}, \cite{FF84}).
The following lemma comes from \cite{W93} and  \cite{LS08}.
\begin{lem}\label{lam-s1}
	For $m\in\Z_{+}$, we have
	
	(1) \ $h^{(m)}_{r,s}=h^{(m)}_{m+2-r,m+3-s}$;
	
	(2) \ For any $(p,q)\neq (r,s)$ and $(p,q)\neq (m+2-r,m+3-s)$, we have
	$$
	h^{(m)}_{r,s}\neq h^{(m)}_{p,q};
	$$
	
	(3) \ For all $(r,s)\neq (1,m+2)$ and $(r,s)\neq (m+1,1)$,
	$$
	h^{(m)}_{r,s}<h^{(m)}_{1,m+2}.$$
\end{lem}

Let $m\in\Z_{+}$. An ordered triple of pairs of integers $((r,s), (r',s'), (r'',s''))$ is called {\em admissible}, if $1\leq r,r',r''\leq m+1$, $1\leq s,s',s''\leq m+2$, $ r+r'+r''\leq 2m+3$, $s+s'+s''\leq 2m+5$, $r<r'+r'', r'<r+r'', r''<r+r', s<s'+s'', s'<s+s'', s''<s+s'$, and both $r+r'+r''$ and $s+s'+s''$ are odd. The following result comes from \cite{W93} ( for $c_1=\frac{1}{2}$, also see \cite{DMZ94}).
\begin{thm}\label{wang}
	Let $m\in\Z_{+}$. Denote
	$$N_{(r,s),(r',s')}^{(r'',s'')}=\dim I_{L(c_{m},0)}\left(\begin{tabular}{c}
	$L(c_{m},h^{(m)}_{r'',s''})$\\
	$L(c_m, h^{(m)}_{r,s})$ $L(c_m, h^{(m)}_{r',s'})$\\
	\end{tabular}\right).$$
	Then $N_{(r,s),(r',s')}^{(r'',s'')}=1$  if and only if $((r,s), (r',s'), (r'',s''))$ is admissible; otherwise, $N_{(r,s),(r',s')}^{(r'',s'')}$ $=0$. In particular,
	$$
	\dim I_{L(c_m,0)}\left(\begin{tabular}{c}
	$L(c_m,0)$\\
	$L(c_m,h^{(m)}_{m+1,1})$ $L(c_m,h^{(m)}_{m+1,1})$\\
	\end{tabular}\right)=1, 
	$$
	$$
	\dim I_{L(c_m,0)}\left(\begin{tabular}{c}
	$L(c_m,h)$\\
	$L(c_m,h^{(m)}_{m+1,1})$ $L(c_m,h^{(m)}_{m+1,1})$\\
	\end{tabular}\right)=0, \ h\neq 0.$$
\end{thm}

Let $V$ be a vertex operator algebra and $\phi$ a conjugate involution of $V$. Recall from \cite{KPP08} and \cite{DL14} that a Hermitian form $(\cdot|\cdot)_H$ on $V$ is called $\phi$-invariant if for any $u,v,w\in V$,
$$
(Y(e^{zL(1)}(-z^{-2})^{L(0)}u,z^{-1})v|w)=(v|Y(\phi(u),z)w).
$$
If further $(\cdot|\cdot)_H$ is positive-definite, $V$ is called unitary.
\section{A class of OZ-type vertex operator algebras}
\def\theequation{3.\arabic{equation}}
\setcounter{equation}{0}
In this section, we will study  vertex operator algebras $V$ satisfying the following conditions:

(I) $V$ is simple, OZ type  and  generated by its Griess algebra $V_2$.

(II) $V_2$ is linearly spanned by Virasoro elements $\omega^{ij}=\omega^{ji}$, $1\leq i<j\leq n$ such that each vertex subalgebra generated by the Virasoro vector $\omega^{ij}$ is isomorphic to the unitary simple vertex operator algebra $L(c_m,0)$ and for distinct $1\leq i,j,k,l\leq n$, 
\begin{equation}\label{e3.1}
 \omega^{ij}_1\omega^{jl}=\frac{1}{2}h^{(m)}_{m+1,1}(\omega^{ij}+\omega^{jl}-\omega^{il}),  
\end{equation}
and
 \begin{equation}\label{e3.2}
\omega^{ij}_p\omega^{kl}=0, \ p\geq 0, \omega^{ij}_3\omega^{ij}=\frac{c_{m}}{2}{\bf 1}, \ \   
 \omega^{ij}_3\omega^{jl}=\frac{c_mh_{m+1,1}^{(m)}}{8}{\bf 1}, 
\end{equation}
where $n\geq 3$,  $m\geq 1$ and $c_m=c_{m+2,m+3}=1-\frac{6}{(m+2)(m+3)}$, $h^{(m)}_{r,s}=\frac{[r(m+3)-s(m+2)]^2-1}{4(m+2)(m+3)}$, $1\leq r\leq m+1, 1\leq s\leq m+1$ are the same as in Section 2. It is easy to see that
 $$h^{(m)}_{m+1,1}=\frac{m(m+1)}{4}.$$
\begin{rmk} If $n=2$, then $V=L(c_m,0)$. 
	If  $m=1$,  $V$ is the  class of vertex operator algebras   generated by  Ising vectors of $\sigma$-type  studied in \cite{JLY19},  \cite{JL16}, \cite{LS08}, \cite{LSY07},  \cite{LY2}, \cite{Ma}, etc. Complete classification of OZ-type  vertex operator algebras generated by Ising vectors of $\sigma$-type was established  in \cite{JLY23}.  
\end{rmk}
\begin{rmk} Let $V$ be an OZ-type vertex operator algebra satisfying (I)-(II), then the Griess algebra $V_2$ is  the non-degenerate  Matsuo algebra $B_{\al,\beta}(G)$
	introduced in \cite{Ma} ( see also \cite{JLY19}) with $\alpha=h^{(m)}_{m+1,1}$, $\beta=c_m$, and $G={\mathfrak S}_n$, where ${\mathfrak S}_n$ is the $n$-symmetric group.
	
\end{rmk}

	Let $V$ be a vertex operator algebra satisfying (I) and (II). Then 
	\begin{equation}\label{e3.5-1}
	\omega^{ij}_1(\omega^{jl}-\omega^{il})={h^{(m)}_{m+1,1}}(\omega^{jl}-\omega^{il}), \  	\omega^{ij}_p(\omega^{jl}-\omega^{il})=0, \ p\in\Z_{\geq 2},
	\end{equation}
	and for distinct $i,j,k,l$, 
	\begin{equation}\label{e3.5-1-2}
\omega^{ij}_p\omega^{kl}=0, \  p\geq 0.
	\end{equation}
Denote by $L^{(ij)}(c_m,0)$ the simple Virasoro algebra generated by $\omega^{ij}$. By (\ref{e3.5-1}), for distinct $1\leq i,j,l\leq n$, 
$\omega^{il}-\omega^{jl}$ generates an $L^{(ij)}(c_m,0)$-module with conformal weight $h_{m+1,1}^{(m)}=\frac{m(m+1)}{4}$. 
Since $V$ is simple and OZ type, it follows from 
\cite{Li94} that  there is a unique non-degenerate bilinear form on $V$ such that 
$$({\bf 1}|{\bf 1})=1$$
and 
$$
(v|Y(u,z)w)=(Y(e^{zL(1)}(-z^{-2})^{L(0)}u,z^{-1})v|w),
$$
for $u,v,w\in V$, where 
$$
Y(w, z)=\sum\limits_{n\in\Z}L(n)z^{-n-2},
$$
and
\begin{equation}\label{e3.3-3}
\omega=\frac{2}{(n-2)h_{m+1,1}^{(m)}+2}\sum\limits_{1\leq i<j\leq n}\omega^{ij}
=\frac{8}{(n-2)m(m+1)+8}\sum\limits_{1\leq i<j\leq n}\omega^{ij}
\end{equation}
is the conformal vector of $V$. 
Then $V$  as a module of itself is isomorphic to its contragredient module $V'=\oplus_{n=0}^{\infty}V_n^*$ \cite{Li94}, where $V_n^*$ is the dual space of $V_n$ \cite{FHL93}.  
By (\ref{e3.2}), 
\begin{equation}\label{e3.3-1}
(\omega^{ij}|\omega^{ij})=\frac{1}{2}c_m=\frac{m(m+5)}{2(m+2)(m+3)}, 
\end{equation}
\begin{equation}\label{eq3.3-2}
(\omega^{ij}|\omega^{jl})=\frac{1}{8}c_mh_{m+1,1}^{(m)}=\frac{m^2(m+1)(m+5)}{32(m+2)(m+3)}.
\end{equation}
For $\omega^{ij}$, denote
$$
Y(\omega^{ij}, z)=\sum\limits_{n\in\Z}L^{(ij)}(n)z^{-n-2}.
$$
Since  $L^{(ij)}(1)\omega^{ij}=0$, it follows that for $u,v\in V$ and $n\in\Z$, 
\begin{equation}\label{e3.3-8}
(L^{(ij)}(n)u|v)=(u|L^{(ij)}(-n)v).
\end{equation}
We  are now in a position to state the main results of this section.
\begin{thm}\label{thm1}
	Let $V$ be  a  vertex operator algebra satisfying (I)-(II). Then $V$ is linearly spanned by elements of the form:
\begin{equation}\label{ethm1-1}
\omega^{i_{1}j_{1}}_{n_{1}}\omega^{i_{2}j_{2}}_{n_{2}}\cdots \omega^{i_{s}j_{s}}_{n_{s}}{\bf 1},
\end{equation}
where $s\geq 0$, $n_{k}\leq 0$ and $0\leq i_{k}<j_{k}\leq n$, $k=1,2,\cdots,s$.
\end{thm}
\begin{proof} By assumption, $V$ is generated by $\omega^{ij}, 1\leq i<j\leq n$ and $V_2$ is linearly spanned by $\omega^{ij}$, $1\leq i<j\leq n$. It follows that $V$ is linearly spanned by elements of the form:
$$
\omega^{i_{1}j_{1}}_{n_{1}}\omega^{i_{2}j_{2}}_{n_{2}}\cdots \omega^{i_{s}j_{s}}_{n_{s}}{\bf 1},	
$$
where $s\geq 0$, $n_{k}\in\Z$ and $0\leq i_{k}<j_{k}\leq n$, $k=1,2,\cdots,s$. For  distinct $i, j, l$ in $\{1,2,\cdots, n\}$, let
\begin{equation}\label{ethm2-6}
\omega^{ijl}=\frac{8}{m(m+1)+8}(\omega^{ij}+\omega^{jl}+\omega^{il}).
\end{equation}
Then
$$
\omega^{ij}_p(\omega^{ijl}-\omega^{ij})=0, \ p\geq 1.
$$
So $\omega^{ijl}-\omega^{ij}$ generates an irreducible $L^{(ij)}(c_m,0)$-module of conformal weight $0$, which is isomorphic to $L^{(ij)}(c_m,0)$ as an $L^{(ij)}(c_m,0)$-module. Thus
$$
\omega^{ij}_0(\omega^{ijl}-\omega^{ij})=0.
$$
By (\ref{e3.5-1}) and (\ref{e3.5-1-2}), $\omega^{jl}-\omega^{il}$ generates an irreducible  $L^{(ij)}(c_m,0)$-module isomorphic to $L^{(ij)}(c_m, h_{m+1,1}^{(m)})$, and $\omega^{kl}$, for $k,l\neq i,j$, generates an irreducible $L(c_m,0)$-module isomorphic to $L(c_m,0)$.  Notice that for fixed $1\leq i<j\leq n$, 
$V_2$ is linearly spanned by $$\{\omega^{ij}, \omega^{il}-\omega^{jl}, \omega^{ijl}-\omega^{ij}, \omega^{kl}|1\leq k<l \leq n, k,l\neq i,j\},$$
and $V$ is generated by $V_2$, and by \cite{W93}
\begin{equation}\label{ethm1-3}
L(c_m, h^{(m)}_{m+1,1})\boxtimes L(c_m, h^{(m)}_{m+1,1})=L(c_m, 0).
\end{equation}
Then we deduce that 
\begin{equation}\label{ethm1-2}
V=V_{[0]}\oplus V_{[h^{(m)}_{m+1,1}]},
\end{equation}
where $V_{[0]}$ is the direct sum of irreducible $L^{(ij)}(c_m,0)$-modules isomorphic to $L^{(ij)}(c_m,0)$, and $V_{[h^{(m)}_{m+1,1}]}$ is the direct sum of irreducible $L^{(ij)}(c_m,0)$-modules isomorphic to $L^{(ij)}(c_m, h^{(m)}_{m+1,1})$.  

\vskip 0.2cm
We denote by $U$ the subspace of $V$ linearly spanned by elements of the form (\ref{ethm1-1}).  As shown in \cite{JL16}, \cite{JLY19}, it is enough to prove that for  fixed $1\leq i<j\leq n$, 
and any homogeneous $u\in U$,  $\omega^{ij}_1u\in U$. We will prove this by induction on the weight of $u$. If $wt(u)=2$, it follows from the assumption that $\omega^{ij}_1u\in U$. Suppose that for $u\in U$ such that $wt(u)\leq n-1$, $\omega^{ij}_1u\in U$. We now assume that $wt(u)=n$, $u\in U$.
  If $u=x_0u^1$, where $u^1\in U$, $x=\omega^{ij}$ or $x=\omega^{ijl}-\omega^{ij}$, for some $l\neq i, j$, then by inductive assumption, 
$$
\begin{array}{ll}
& \omega^{ij}_1u=\omega^{ij}_1\omega^{ij}_0u^1\\
&=\omega^{ij}_0\omega^{ij}_1u^1+\sum\limits_{j=0}^1\left(\begin{array}{l}1 \\ j\end{array}\right)(\omega^{ij}_j\omega^{ij})_{1-j}u^1\\
&=\omega^{ij}_0\omega^{ij}_1u^1+\omega^{ij}_0u^1\in U
\end{array}
$$
if $x=\omega^{ij}$, and 
$$
\begin{array}{ll}
& \omega^{ij}_1u=\omega^{ij}_1(\omega^{ijl}-\omega^{ij})_0u^1\\
&=(\omega^{ijl}-\omega^{ij})_0\omega^{ij}_1u^1+\sum\limits_{j=0}^1\left(\begin{array}{l}1 \\ j\end{array}\right)(\omega^{ij}_j(\omega^{ijl}-\omega^{ij}))_{1-j}u^1\\
&=(\omega^{ijl}-\omega^{ij})_0\omega^{ij}_1u^1\in U
\end{array}
$$
if $x=\omega^{ijl}-\omega^{ij}$. So we may assume that $u=(\omega^{jl}-\omega^{il})_0u^1$, for some $l\neq i,j$.  By (\ref{ethm1-2}), we may assume that $u^1\in V_{[0]}$ or $u^1\in V_{[h^{(m)}_{m+1,1}]}$. 

\vskip 0.2cm
{\bf Case 1} $u^1\in V_{[0]}$.  By inductive assumption, $V_k=U_k$, for $0\leq k\leq n-1$. Then we may assume that 
$$
u^1=u^2+ u^3,
$$
where $u^2$ is a highest weight vector of an irreducible $L(c_m,0)$-module isomorphic to $L(c_m,0)$, and $u^3$ is a sum of elements of the form
$$
u^3=\sum\limits_{k=1}^s\omega^{ij}_{-m_k}v^k, 
$$
where $m_1, \cdots, m_s\geq 1$, $v^1, \cdots, v^s\in U$. Then 
$$
\omega^{ij}_pu^2=0, \ p\geq 0, 
$$
and 
$$
\begin{array}{ll}
&\omega^{ij}_1(\omega^{jl}-\omega^{il})_0u^2\\
&=(\omega^{ij}_0(\omega^{jl}-\omega^{il}))_1u^2+h^{(m)}_{m+1,1}(\omega^{jl}-\omega^{il})_0u^2\\
&=\omega^{ij}_0(\omega^{jl}-\omega^{il})_1u^2+h^{(m)}_{m+1,1}(\omega^{jl}-\omega^{il})_0u^2\in U,
\end{array}
$$
$$
\begin{array}{ll}
&\omega^{ij}_1(\omega^{jl}-\omega^{il})_0u^3=\omega^{ij}_1(\omega^{jl}-\omega^{il})_0(\sum\limits_{k=1}^s\omega^{ij}_{-m_k}v^k)\\
&=(\omega^{jl}-\omega^{il})_0\omega^{ij}_1u^3+h^{(m)}_{m+1,1}(\omega^{jl}-\omega^{il})_0u^3+\sum\limits_{k=1}^s(\omega^{ij}_0(\omega^{jl}-\omega^{il}))_1\omega^{ij}_{-m_k}v^k\\
& =(\omega^{jl}-\omega^{il})_0\omega^{ij}_1u^3+h^{(m)}_{m+1,1}(\omega^{jl}-\omega^{il})_0u^3+\sum\limits_{k=1}^s\omega^{ij}_{-m_k}(\omega^{ij}_0(\omega^{jl}-\omega^{il}))_1v^k\\
&+((\omega^{ij}_0(\omega^{jl}-\omega^{il}))_0\omega^{ij})_{-m_k+1}v^k+h^{(m)}_{m+1,1}(\omega^{ij}_0(\omega^{jl}-\omega^{il}))_{-m_k}v^k.
\end{array}
$$
Since $m_k\geq 1$, together with the inductive assumption, we deduce that $\omega^{ij}_1(\omega^{jl}-\omega^{il})_0u^3\in U$.

\vskip 0.2cm
{\bf Case 2} $u^1\in V_{[h^{(m)}_{m+1,1}]}$. By (\ref{ethm1-3}),  $(\omega^{jl}-\omega^{il})_0u^1\in V_{[0]}$. Since $V_{[0]}$ is a direct sum of irreducible $L(c_m,0)$-modules isomorphic to $L(c_m,0)$, it follows that for any $w\in V_n$, $w=w^1+w^2$, where $w^1$ is a highest weight vector of an irreducible $L(c_m,0)$-module isomorphic to $L(c_m,0)$, and $w^2$ is a sum of elements of the form: $\omega^{ij}_{-n_1}\cdots \omega^{ij}_{-n_r}w'$ such that  $n_k\geq 1$ and $w'$ is a highest weight vector of an irreducible $L(c_m,0)$-module isomorphic to $L(c_m,0)$. So 
$$
(\omega^{jl}-\omega^{il})_0u^1=u'+u'',
$$
where $u'$ is a highest weight vector of an irreducible $L(c_m,0)$-module isomorphic to $L(c_m,0)$, and $u''$ is a sum of elements of the form: $\omega^{ij}_{-m_1}\cdots \omega^{ij}_{-m_s}y$ such that 
$m_k\geq 1$ and $y$ is a highest weight vector of an irreducible $L(c_m,0)$-module isomorphic to $L(c_m,0)$. By inductive assumption, $\omega^{ij}_{-m_1}\cdots \omega^{ij}_{-m_s}y\in U$. So $u''\in U$. Then we have
$$
\omega^{ij}_1(\omega^{jl}-\omega^{il})_0u^1=\omega^{ij}_1(u'+u'')=\omega^{ij}_1u''.
$$
Notice that 
$$
\omega^{ij}_1\omega^{ij}_{-m_1}\cdots \omega^{ij}_{-m_s}y=\omega^{ij}_{-m_1}\omega^{ij}_1\omega^{ij}_{-m_2}\cdots \omega^{ij}_{-m_s}y+(m_1+1)\omega^{ij}_{-m_1}\cdots \omega^{ij}_{-m_s}y\in U.
$$
This proves that $\omega^{ij}_1(\omega^{jl}-\omega^{il})_0u^1\in U$. 
\end{proof}
\begin{thm}\label{thm2}
	Let $V$ be  an OZ-type  vertex operator algebra satisfying (I)-(II). Then $V$ is uniquely determined by the structure of $V_2$. 	
\end{thm}
\begin{proof}
	Let $\widetilde{V}$ be another simple OZ-type vertex operator algebra generated by Virasoro vectors $\widetilde{\omega}^{ij}$, $1\leq i<j\leq n$ such that each vertex subalgebra generated by the Virasoro vector $\widetilde{\omega}^{ij}$ is isomorphic to the unitary simple vertex operator algebra $L(c_m,0)$ and for distinct $1\leq i,j,l\leq n$, 
	\begin{equation}\label{ethm2-1}
	\widetilde{\omega}^{ij}_1\widetilde{\omega}^{jl}=\frac{h^{(m)}_{m+1,1}}{2}(\widetilde{\omega}^{ij}+\widetilde{\omega}^{jl}-\widetilde{\omega}^{il}).
	\end{equation}
	and 
	\begin{equation}\label{ethm2-2} 
	\widetilde{\omega}^{ij}_3\widetilde{\omega}^{ij}=\frac{c_{m}}{2}{\bf 1}, \ \   
	\widetilde{\omega}^{ij}_3\widetilde{\omega}^{jl}=\frac{c_mh_{m+1,1}^{(m)}}{8}{\bf 1}.
	\end{equation}
Let $(\cdot|\cdot)$ be the unique non-degenerate bilinear form on $\widetilde{V}$ such that $({\bf 1}|{\bf 1})=1$. It is obvious that for $1\leq i_1<j_2\leq n$, $1\leq i_2<j_2\leq n$, 
\begin{equation}\label{ethm2-3}
(\omega^{i_1j_1}|\omega^{i_2j_2})=(\widetilde{\omega}^{i_1j_1}|\widetilde{\omega}^{i_2j_2}).
\end{equation}
Since for $1\leq i<j\leq n$, $m\in\Z$,  and $u,v\in V$, $\widetilde{u}, \widetilde{v}\in\widetilde{V}$, 
\begin{equation}\label{ethm2-4}
(\omega^{ij}_mu|v)=(u|\omega^{ij}_{-m+2}v), \  (\widetilde{\omega}^{ij}_m\widetilde{u}|\widetilde{v})=(\widetilde{u}|\widetilde{\omega}^{ij}_{-m+2}\widetilde{v}).
\end{equation}
Let $u$ be a linear combination of elements of the form: $\omega^{i_{1}j_{1}}_{n_{1}}\omega^{i_{2}j_{2}}_{n_{2}}\cdots \omega^{i_{s}j_{s}}_{n_{s}}{\bf 1}$, then we  denote by $\widetilde{u}$  the same linear combination of elements of the form: $\widetilde{\omega}^{i_{1}j_{1}}_{n_{1}}\widetilde{\omega}^{i_{2}j_{2}}_{n_{2}}\cdots \widetilde{\omega}^{i_{s}j_{s}}_{n_{s}}{\bf 1}$, 
where $s\geq 0$, $n_{k}\leq 0$ and $0\leq i_{k}<j_{k}\leq n$, $k=1,2,\cdots,s$.
\vskip 0.2cm
Based on (\ref{ethm2-3}) and (\ref{ethm2-4}), by Theorem \ref{thm1} and induction assumption on the weights of $u,v\in V$ and $\widetilde{u}, \widetilde{v}\in \widetilde{V}$, we can deduce that 
\begin{equation}\label{ethm2-5}
(u|v)=(\widetilde{u}|\widetilde{v}).
\end{equation}
(\ref{ethm2-5})  implies that $u=0$ if and only if $\widetilde{u}=0$, since both $V$ and $\widetilde{V}$ are simple. Thus 
the  linear map $\phi: V_2\to \widetilde{V}_2$ defined by 
$$
\phi(\omega^{ij})=\widetilde{\omega}^{ij},  \  1\leq i<j\leq n.
$$
can be lifted to a vertex algebra homomorphism from $V$ to $\widetilde{V}$. Similarly, the linear map $\phi: \widetilde{V}_2\to {V}_2$ defined by 
$$
\widetilde{\phi}(\widetilde{\omega}^{ij})={\omega}^{ij},  \  1\leq i<j\leq n.
$$
can also be lifted to a vertex algebra homomorphism from $\widetilde{V}$ to ${V}$. It is obvious that $\phi^{-1}=\widetilde{\phi}$. Thus $V\cong \widetilde{V}$. This proves that $V$ is uniquely determined by the structure of $V_2$.
\end{proof}
\begin{cor}\label{cor1}
	 	Let $V$ be  an OZ-type  vertex operator algebra satisfying (I)-(II). Then the Zhu algebra $A(V)$ of $V$ is generated by $[\omega^{ij}]$, $1\leq i<j\leq n$.
\end{cor}
\begin{proof}
	The proof  is similar to that of Theorem 3.7 in \cite{JL16}. Let $A'$ be the subalgebra of $A(V)$ generated by $[\omega^{ij}]$, $0\leq i<j\leq n$. It suffices to show that for every homogeneous $u\in V$, $[u]\in A'$. We will approach it by induction on the weight of $u$. If $wt(u)=2$, it is obvious that $[u]\in A'$. Suppose that for all homogeneous $u$ such that $wt(u)\leq q-1$, we have $[u]\in A'$. Now suppose $wt(u)=q$. By Theorem \ref{thm1}, we may assume that
	$$
	u=\omega^{i_{1}j_{1}}_{n_{1}}\omega^{i_{2}j_{2}}_{n_{2}}\cdots \omega^{i_{s}j_{s}}_{n_{s}}{\bf 1},
	$$
	for some $n_{k}\leq 0$ and $0\leq i_{k}<j_{k}\leq n$, $k=1,2,\cdots,s$. Since for $p\geq 0$, $[(\omega^{i_{1}j_{1}}_{-p-2}+2\omega^{i_{1}j_{1}}_{-p-1}+\omega^{i_{1}j_{1}}_{-p})
	\omega^{i_{2}j_{2}}_{n_{2}}\cdots \omega^{i_{s}j_{s}}_{n_{s}}{\bf 1}]=[0]$, we may assume that $-1\leq n_{1}\leq 0$. Denote $u^2=\omega^{i_{2}j_{2}}_{n_{2}}\cdots \omega^{i_{s}j_{s}}_{n_{s}}{\bf 1}$. Then by inductive assumption, $[u^2], [\omega^{i_{1}j_{1}}_{1}u^2]\in A'$.  Since
	$$
	[(\omega^{i_{1}j_{1}}_{0}+\omega^{i_{1}j_{1}}_{1})u^2]=[\omega^{i_{1}j_{1}}]*[u^2]-[u^2]*[\omega^{i_{1}j_{1}}],$$
	it follows that $[\omega^{i_{1}j_{1}}_{0}u^2]\in A'$. Then by the fact that
	$$[\omega^{i_{1}j_{1}}]*[u^2]=[(\omega^{i_{1}j_{1}}_{-1}+2\omega^{i_{1}j_{1}}_{0}+\omega^{i_{1}j_{1}}_{1})u^2],$$
	we have $[u]\in A'$. Therefore $A(V)=A'$.
\end{proof}
We now  recall the  vertex operator algebra $M(A_{n-1})$ for $n\geq 2$ from \cite{LS08},   \cite{JL16}, \cite{CL07},  \cite{LSY07},  \cite{JLY19}, \cite{JLY23}, etc. 

\vskip 0.2cm 
Let $L(sl_2,k)$ be the simple affine vertex operator algebra associated to the simple Lie algebra $sl_2$ and level $k$.  For $n\geq 2$, let $L(sl_2, 1)^{\otimes n}$ be the tensor product of the vertex operator algebra $L(sl_2,1)$. It is known that 
$L(sl_2,1)^{\otimes n}$ is isomorphic to the lattice vertex operator algebra $V_{L}$, where $L=\oplus_{i=1}^n\Z\alpha_i$ such that $(\alpha_i|\alpha_j)=2\delta_{ij}$. For $1\leq i,j\leq n$, $i\neq j$, denote
\begin{equation}\label{ethm3-1}
\omega^{ij}=\frac{1}{16}(\al^i-\al^j)(-1)(\al^i-\al^j)(-1){\bf 1}-\frac{1}{4}(e^{\al^i-\al^j}+e^{-\al^i+\al^j}).
\end{equation}
 It is obvious that $L(sl_2,n)$ can be diagonally embedded into $L(sl_2, 1)^{\otimes n}$.  Denote by $M(A_{n-1})$ the commutant of $L(sl_2,n)$ in $L(sl_2, 1)^{\otimes n}$.  Then $M(A_{n-1})$  is generated by the $\omega^{ij}$, $1\leq i<j\leq n$, given by (\ref{ethm3-1}), such that $\omega^{ij}$ generates a simple Virasoro vertex operator algebra isomorphic to $L(1/2,0)$ and for distinct $1\leq i,j,l\leq n$, $\omega^{ij}, \omega^{jl}, \omega^{il}$ satisfy (\ref{e3.1}) and (\ref{e3.2}) with $m=1$.  We now assume that $n=3$. For distinct $1\leq i,j\leq 3$, let $\omega^{ij}\in M(A_2)$ be given by (\ref{ethm3-1}). 
 Then \cite{LS08}
 \begin{equation}\label{ethm3-2}
 M(A_2)=L^{(ij)}(1/2, 0)\otimes L(7/10, 0)\oplus L^{(ij)}(1/2, 1/2)\otimes L(7/10, 3/2).
 \end{equation}
 The conformal vector of $M(A_2)$ is 
 \begin{equation}\label{ethm3-3}
 \omega=\frac{4}{5}(\omega^{12}+\omega^{23}+\omega^{13}).
 \end{equation}
 For $1\leq i,j\leq 3$, set
 $$
 \widetilde{\omega}^{ij}=\omega-\omega^{ij}.
 $$
 Then $M(A_2)$ is generated by $\widetilde{\omega}^{ij}$, $1\leq i<j\leq 3$, and  by (\ref{ethm3-2}) and (\ref{ethm3-3}), $\widetilde{\omega}^{ij}$ generates a simple Virasoro vertex operator isomorphic to $L(7/10,0)$. Furthermore, it is easy to check that for distinct $1\leq i,j,l\leq 3$,
  $$
  	\widetilde{\omega}^{ij}_3\widetilde{\omega}^{ij}=\frac{c_{m}}{2}{\bf 1}, 
\widetilde{\omega}^{ij}_3\widetilde{\omega}^{jl}=\frac{c_mh_{m+1,1}^{(m)}}{8}{\bf 1}
  $$
 with $m=2$. Then by Theorem \ref{thm1} we immediately have
\begin{thm}\label{thm3}
	Let $V$ be  a vertex operator algebra satisfying (I)-(II). If $n=3$ and $m=2$, then $V\cong M(A_2)$.
\end{thm}
\begin{rmk}
	It is interesting to give more  examples of vertex operator algebras satisfying (I)-(II). 
	\end{rmk}

Let $V$ be an OZ-type vertex operator algebra satisfying (I) and (II).  For $1\leq i<j\leq n$, let $L^{(ij)}(c_m,0)$ be the simple Virasoro vertex operator algebra generated by $\omega^{ij}$. By the assumption, 
$$
V=V_{[0]}\oplus V_{[h^{(m)}_{m+1,1}]}, 
$$
where $V_{[0]}$ is the direct sum of irreducible $L^{(ij)}(c_m,0)$-modules isomorphic to $L^{(ij)}(c_m,0)$, and $V_{[h^{(m)}_{m+1,1}]}$ is the direct sum of irreducible $L^{(ij)}(c_m,0)$-modules isomorphic to $L^{(ij)}(c_m, h^{(m)}_{m+1,1})$.  By the fusion rules of the Virasoro algebra $L(c_m,0)$ \cite{W93}, 
$$
 V_{[h^{(m)}_{m+1,1}]}\boxtimes  V_{[h^{(m)}_{m+1,1}]}\subseteq V_{[0]}.
$$
Define linear map $\sigma^{ij}: V\to V$ by 
\begin{equation}\label{eauto1}
\sigma^{ij}_{V_{[0]}}={\rm id}_{V_{[0]}}, \ \sigma^{ij}_{V_{[h^{(m)}_{m+1,1}]}}=-{\rm id}_{V_{[h^{(m)}_{m+1,1}]}}.
\end{equation}
Then by (\ref{eauto1}), $\sigma^{ij}$ is an involution of $V$. 
\begin{rmk}
	For $m=1$, the involution $\sigma^{ij}$ was first introduced  in \cite{Mi96}, usually called Miyamoto involution of $\sigma$-type. 
\end{rmk}
The following result is to characterize the automorphism group of $V$.
\begin{thm}\label{thm4}
	Let $V$ be an OZ-type vertex operator algebra satisfying (I) and (II). Then ${\rm Aut}V$ is isomorphic to the symmetric group ${\mathfrak S}_n$.
\end{thm}
\begin{proof}
For $1\leq i<j\leq n$, let $\sigma^{ij}$ be the involution of $V$ defined by (\ref{eauto1}). For $k, l\neq i,j$, since $\omega^{il}-\omega^{jl}\in V_{[h^{(m)}_{m+1,1}]}$, $\omega^{ijl}-\omega^{ij}$, 
$\omega^{kl}\in V_{[0]}$, it follows that 
\begin{equation}\label{eauto2}
\sigma^{ij}(\omega^{il}-\omega^{jl})=-(\omega^{il}-\omega^{jl}),  \ \sigma^{ij}(\omega^{ijl}-\omega^{ij})=\omega^{ijl}-\omega^{ij},  \ \sigma^{ij}\omega^{kl}=\omega^{kl},
\end{equation}
where $\omega^{ijl}$ is defined by (\ref{ethm2-6}).  Notice that $\omega^{ij}\in V_{[0]}$, thus for $l\neq i,j$, 
\begin{equation}\label{eauto3}
\sigma^{ij}\omega^{il}=\omega^{jl}, \ \sigma^{ij}\omega^{jl}=\omega^{il}.
\end{equation}
Then it is easy to check that for distinct $1\leq i,j,k,l\leq n$, 
\begin{equation}\label{eauto4}
\sigma^{ij}\sigma^{kl}=\sigma^{kl}\sigma^{ij}, \ \sigma^{ij}\sigma^{jk}\sigma^{ij}=\sigma^{ik}=\sigma^{ik}\sigma^{ij}\sigma^{jk}.
\end{equation}
This proves that the automorphic subgroup of $V$ generated by $\omega^{ij}, 1\leq i<j\leq n$ is isomorphic to the symmetric group ${\mathfrak S}_n$. Thus
${\mathfrak S}_n\subseteq {\rm Aut}V$. 

\vskip 0.2cm
 Conversely, let $\sigma\in{\rm Aut}V$, then $\sigma(V_2)=V_2$, and for every Virasoro vector $\omega^{ij}$, $\sigma(\omega^{ij})$ is still a Virasoro vector of central charge $c_m$. Since $V$ is generated by $V_2$, $\sigma$ is uniquely determined by $\sigma|_{V_2}$. 
 
 \vskip 0.2cm
 \noindent
  Let $\Delta=\{\epsilon_i-\epsilon_j| 1\leq i, j\leq n, i\neq j\}$ be the root system of $sl_n$ such that $(\epsilon_i|\epsilon_j)=\delta_{ij}$.  Notice that $\{\omega^{ij}, 1\leq i<j\leq n\}$ exhaust all the Virasoro vectors of central charge $c_m$, and each $\omega^{ij}$ corresponds to the root $\epsilon_i-\epsilon_j$. Since $\sigma(\omega^{ij})$ is still a Virasoro vector, it follows that there exists the corresponding $\widetilde{\sigma}: \Delta\to \Delta$ such that 
  $$
  \sigma(\omega^{ij})=\omega^{kl}, \ 
  $$
 if and only if  $\widetilde{\sigma}(\epsilon_i-\epsilon_j)=\epsilon_k-\epsilon_l$.  By (\ref{e3.2}), 
 $$
 (\epsilon_i-\epsilon_j|\epsilon_r-\epsilon_s)=(\widetilde{\sigma}(\epsilon_i-\epsilon_j)|\widetilde{\sigma}(\epsilon_r-\epsilon_s)).
 $$
 This means that $\widetilde{\sigma}$ is an automorphism of the root system $\Delta$. It is known that the automorphic group of $\Delta$ is isomorphic to ${\mathfrak S}_n$. This means that $\sigma\in{\mathfrak S}_n$.  Thus ${\rm Aut}V={\mathfrak  S}_n$.
\end{proof}

\section{Conditions for Positivity}

Let $V$ be an OZ-type vertex operator algebra satisfying (I)-(II).  Since $V$ is simple and OZ type, it follows from 
\cite{Li94} that  there is a unique non-degenerate bilinear form on $V$ such that 
$$({\bf 1}|{\bf 1})=1$$
and 
$$
(v|Y(u,z)w)=(Y(e^{zL(1)}(-z^{-2})^{L(0)}u,z^{-1})v|w),
$$
for $u,v,w\in V$.
For $\omega^{i_1j_1}_{m_1}\cdots \omega^{i_sj_s}_{m_s}{\bf 1}\in V$, and $a\in\C$,  define 
$$\sigma(a\omega^{i_1j_1}_{m_1}\cdots \omega^{i_sj_s}_{m_s}{\bf 1})=\bar{a}\omega^{i_1j_1}_{m_1}\cdots \omega^{i_sj_s}_{m_s}{\bf 1}.
$$
Then $\sigma$ can be extended to a conjugate involution of $V$. 
 We have  the associated $\sigma$-invariant Hermitian form $(\cdot|\cdot)_H$ on $V$ such that 
\begin{equation}\label{e3.4}
(au|bv)_H=a\bar{b}(u|v),
\end{equation}
for $a, b\in\C$, and $u=\omega^{i_1j_1}_{m_1}\cdots \omega^{i_sj_s}_{m_s}{\bf 1}$, $v=\omega^{r_1l_1}_{n_1}\cdots \omega^{r_tl_t}_{n_t}{\bf 1}$, where $1\leq i_k< j_k\leq n, 1\leq r_p<l_p\leq n$, $m_k, n_p\in\Z_{\geq 0}$, $k=1, 2,\cdots, s,  p=1, 2, \cdots, t$.  It is obvious that
\begin{equation}\label{e3.5}
(\omega^{ij}_nu|v)_H=(u|\omega^{ij}_{-n+2}v)_H.
\end{equation}

We  have the following proposition.
\begin{prop} Let $V$ be a vertex operator algebra satisfying (I) and (II), and $n\geq 3$, $m\geq 2$. Then the restriction of Hermitian form $(\cdot|\cdot)_{H}$ on $V_2$ is positive definite if and only if $m\leq 3$ if $n=3$, and $m=2$ if $n\geq 4$. 
\end{prop}
\begin{proof}  
	 Let
	 $u^1=\omega^{12}, u^2=\omega^{23}, u^3=\omega^{13}$, $a_{ij}=(u^i|u^j)_H$,  and $A=(a_{ij})_{3\times 3}$. Then 
	$$
	A=\frac{m(m+5)}{2(m+2)(m+3)}\left(\begin{smallmatrix} 1 & \frac{m(m+1)}{16} & \frac{m(m+1)}{16} \\
	\frac{m(m+1)}{16} & 1 & \frac{m(m+1)}{16}\\
	\frac{m(m+1)}{16} & \frac{m(m+1)}{16} & 1\\
	\end{smallmatrix}\right)
	$$
It is easy to see that if $A$ is positive-definite, then $m\leq 3$, and for $n=3$, $A$ is positive-definite if and only if $m\leq 3$.

 We now assume that $n\geq 4$. 
Suppose the statement is true for  $n-1$ case, we now consider the $n$ case. 
 We consider $V$ as a module of $L^{(12)}(c_m,0)$. Then $V$ is completely reducible.

Let $U_2$ be the subspace of $V_2$ linearly spanned by $\omega^{ij}, 3\leq i<j\leq n$. Since 
$$
\omega^{12}_p\omega^{kl}=0, 
$$
for $p\geq 2$, $3\leq k<l\leq n$, it follows that any element in $U_2$ is a primary vector of $L^{(12)}(c_m,0)$ of weight 0. In particular, 
$$
(\omega^{12}|U_2)=0. 
$$
By inductive assumption, the Hermitian form $(\cdot|\cdot)_H$ restricted on $U_2$ is positive definite.  So  there exists an orthonormal basis in $U_2$. Notice that
$$
(\omega^{12}|\omega^{1j}-\omega^{2j})=0, \   (\omega^{kl}|\omega^{1j}-\omega^{2j})=0, 
$$
for $3\leq j\leq n$, $3\leq k<l\leq  n$, and $\omega^{1k}-\omega^{2k}$ for $3\leq k\leq n$ are primary vectors of $L^{(12)}(c_m,0)$ of weight $h_{m+1,1}^{(m)}$. It is easy to see that 
$$
V_2={\mathbb C}\omega^{12}\oplus U_2\oplus \sum\limits_{i=3}^n{\mathbb C}(\omega^{1k}-
\omega^{2k})\oplus \sum\limits_{i=3}^n{\mathbb C}(\omega^{1k}+\omega^{2k}-\frac{1}{2}h_{m+1,1}^{(m)}\omega^{12}), 
$$
and $\omega^{1k}+\omega^{2k}-\frac{1}{2}h_{m+1,1}^{(m)}\omega^{12}$ for $3\leq k\leq n$ are primary vectors of $L^{(12)}(c_m,0)$ of weight 0. So
$$
(\omega^{1k}-\omega^{2k}|\omega^{1j}+\omega^{2j}-\frac{1}{2}h_{m+1,1}^{(m)}\omega^{12})=0,
$$
for $3\leq k,j\leq n$. Notice that for $3\leq k, l\leq n$, $k\neq l$, 
\begin{equation}\label{e3.4-1}
(\omega^{1k}-\omega^{2k}|\omega^{1k}-\omega^{2k})=\frac{m(m+5)}{(m+2)(m+3)}(1-\frac{m(m+1)}{16}), 
\end{equation}
\begin{equation}\label{e3.4-2}
(\omega^{1k}-\omega^{2k}|\omega^{1l}-\omega^{2l})=\frac{m^2(m+1)(m+5)}{16(m+2)(m+3)},
\end{equation}
\begin{eqnarray}
& (\omega^{1k}+\omega^{2k}-\frac{1}{2}h_{m+1,1}^{(m)}\omega^{12}|\omega^{1k}+\omega^{2k}-\frac{1}{2}h_{m+1,1}^{(m)}\omega^{12})\nonumber\\
& \nonumber\\
&=\frac{m(m+5)}{(m+2)(m+3)}+\frac{m^2(m+1)(m+5)}{16(m+2)(m+3)}[1-\frac{m(m+1)}{8}]\label{e3.4-3},
\end{eqnarray}
\begin{equation}\label{e3.4-4}
(\omega^{1k}+\omega^{2k}-\frac{1}{2}h_{m+1,1}^{(m)}\omega^{12}|\omega^{1l}+\omega^{2l}-\frac{1}{2}h_{m+1,1}^{(m)}\omega^{12})=\frac{m^2(m+1)(m+5)}{16(m+2)(m+3)}[1-\frac{m(m+1)}{8}].
\end{equation}
For $3\leq s\leq n$,  $3\leq k,l\leq s$, let 
$$
b_{kl}=(\omega^{1k}-\omega^{2k}|\omega^{1l}-\omega^{2l}), $$
$$ \ c_{kl}=(\omega^{1k}+\omega^{2k}-\frac{1}{2}h_{m+1,1}^{(m)}\omega^{12}|\omega^{1l}+\omega^{2l}-\frac{1}{2}h_{m+1,1}^{(m)}\omega^{12}),
$$
and
$$
B=(b_{kl})_{(s-2)\times (s-2)}, \ \ C=(c_{kl})_{(s-2)\times (s-2)}.
$$
Then 
$$
\begin{array}{ll}
\det B&=[\frac{m(m+5)}{(m+2)(m+3)}]^{s-2}\det \left(\begin{smallmatrix} 1-\frac{m(m+1)}{16} & \frac{m(m+1)}{16} & \cdots &  \frac{m(m+1)}{16} \\
\frac{m(m+1)}{16} & 1-\frac{m(m+1)}{16} &  \cdots & \frac{m(m+1)}{16}\\
\cdots & \cdots & \cdots & \cdots \\
\frac{m(m+1)}{16} & \frac{m(m+1)}{16} & \cdots &1-\frac{m(m+1)}{16}\\
\end{smallmatrix}\right)\\\\
&=[\frac{m(m+5)}{(m+2)(m+3)}]^{s-2}[1-\frac{m(m+1)}{8}]^{s-3}[1+\frac{(s-4)m(m+1)}{16}],
\end{array}
$$
and
$$
\det C=[\frac{m(m+5)}{(m+2)(m+3)}]^{s-3}[\frac{m(m+5)}{(m+2)(m+3)}+\frac{(s-2)m^2(m+1)(m+5)}{16(m+2)(m+3)}(1-\frac{m(m+1)}{8})].
$$
$(\cdot|\cdot)_{H}$ on $V_2$ is positive-definite  if and only if $\det B>0$ and $\det C>0$ for all $3\leq s\leq n$. Since $n\geq 4$, we see that $(\cdot|\cdot)_{H}$ on $V_2$ is positive-definite  if and only if $1\leq m\leq 2$.
\end{proof}
\begin{thm} Let $V$ be a vertex operator algebra satisfying the conditions (I) and (II). If $V$ is unitary, 
	then $V$ is $C_2$-cofinite.
\end{thm}
\begin{proof}
	The idea of our proof comes from \cite{Li99} and  \cite{DW11}. 
 Denote by $\overline{V}$ the completion of $V$, which is the direct product of $V_n$, $n\in\Z_{\geq 0}$. Then each element in $\overline{V}$ has the form:
$$
(v^0,v^1, ... ,v^n, ...), 
$$
where $v_i\in V_i$. $\overline{V}$ is isomorphic to the dual space $V^*$ of $V$ via the non-degenerate bilinear $(\cdot|\cdot)$. Let $t$ be an indeterminate and 
$$
L(V)=V\otimes {\mathbb C}[t,t^{-1}]. 
$$
Then \cite{Bor86}
$$
\widehat{V}=L(V)/DL(V)
$$
is Lie algebra with bracket
$$
[a(p), b(q)]=\sum\limits_{i=0}^{\infty}\left(\begin{array}{c}p\\i\end{array}\right)(a_ib)_{p+q-i},
$$
where $D=\frac{d}{dt}\otimes {\bf 1}+{\bf 1}\otimes L(-1)$ and $v(m)$ is the image of $v\otimes t^m$ in $\widehat{V}$ for $v\in V$ and $m\in\Z$.  Then $\overline{V}$ is a module of the Lie algebra $\widehat{V}$ \cite{FHL93}, \cite{Li99}. Let $D(V)$ be the subspace of $\overline{V}$ consisting of vectors $v$ such that 
$u_nv=0$, for $u\in V$ and $n$ sufficiently large. Then $D(V)$ is a weak $V$-module and $V$ is a submodule of $D(V)$ \cite{Li99}. 
 By \cite{Li99} if $D(V)=V$, then $V$ is $C_2$-cofinite. Suppose that $D(V)\neq V$. Then there exists  $v=(v^0,v^1, ... ,v^n, ...)$ such that $v\notin V$. This implies that there are infinitely many non-zero $v^i$.  For each $\omega^{ij}$, denote by $L^{ij}(c_m,0)$ the Virasoro vertex  subalgebra of $V$ generated by $\omega^{ij}$. Since $L^{ij}(c_m,0)$ is rational and $C_2$-cofinite with non-negative conformal weights, it follows that $D(V)$ is a module of $L^{ij}(c_m,0)$. So $v=(v^0,v^1, ... ,v^n, ...)$ is a direct sum of eigenvectors of $L^{ij}(0)$ with non-negative eigenvalues $ \lambda^{ij}_1, \cdots, \lambda^{ij}_{s_{ij}}$. Notice that $L^{ij}(0)V_n\subseteq V_n$, for each $n\in\Z_{\geq 0}$. So for each $i\in\Z_{\geq 0}$, $v^i$ is a direct sum of eigenvectors of $L^{ij}(0)$ with possible eigenvalues  $\lambda^{ij}_1,\cdots,\lambda^{ij}_{s_{ij}}$. Let $\lambda^{ij}={\rm \max}\{\lambda^{ij}_1,\cdots,\lambda^{ij}_{s_{ij}}\}$. Then 
we have 
$$
(L^{ij}(0)v^i|v^i)_{H}\leq \lambda^{ij}(v^i|v^i)_H,
$$
for all $i\in\Z_{\geq 0}$. Notice that 
$$
\omega=\frac{8}{(n-2)m(m+1)+8}\sum\limits_{1\leq i<j\leq n}\omega^{ij}.
$$
So we have 
$$
\begin{array}{ll}
(L(0)v^i|v^i)_H
=i(v^i|v^i)_H=\frac{8}{(n-2)m(m+1)+8}\sum\limits_{1\leq i<j\leq n}(L^{ij}(0)v^i|v^i)_H\\
\leq \frac{8}{(n-2)m(m+1)+8}\sum\limits_{1\leq i<j\leq n}\lambda^{ij}(v^i|v^i)_H, 
\end{array}
$$
for all $i\in\Z_{\geq 0}$, which is impossible. 
\end{proof}

\end{document}